\documentclass[12pt]{amsart}
\usepackage[latin1]{inputenc}
\usepackage{latexsym}
\usepackage{amsfonts}
\usepackage{amsmath,amssymb}
\usepackage{amscd}
\usepackage[all]{xy}

\def\var{\varepsilon}

\theoremstyle{plain}
\newtheorem{theorem}{Theorem}[section]
\newtheorem{lemma}[theorem]{Lemma}
\newtheorem{proposition}[theorem]{Proposition}
\newtheorem{corollary}[theorem]{Corollary}

\theoremstyle{definition}
\newtheorem{example}[theorem]{Example}
\newtheorem{definition}[theorem]{Definition}

\theoremstyle{remark}
\newtheorem*{remark}{Remark}
\begin{document}
\title[(m)-condition and Strong Milnor Fibrations ]{Uniform (m)-condition and Strong Milnor Fibrations }

\author[Ara\'ujo dos Santos]{Raimundo N. Ara\'ujo dos Santos}
\thanks{The author acknowledges the financial support from FAPESP/Brazil grant $\#$ 05/58953-7}
\email[ Ara\'ujo dos Santos, Raimundo N.]{rnonato@icmc.usp.br}
\address{\hspace{-0.4cm}Instituto de Ci\^encias Matem\'aticas e de Computa\c c\~ao\newline
Universidade de S\~ao Paulo\\ Av. Trabalhador S\~ao-Carlense, 400 -
Centro  Postal Box 668, S\~ao Carlos - S\~ao Paulo - Brazil
\newline Postal Code 13560-970, S\~ao Carlos, SP, Brazil.}





\begin{abstract}

In this paper we study the Milnor fibrations associated to real
analytic map germs $\psi:(\mathbb{R}^{m},0) \to (\mathbb{R}^2,0)$
with isolated critical point at $0\in \mathbb{R}^{m}$. The main
result relates the existence of called Strong Milnor fibrations
with a transversality condition of a convenient family of analytic
varieties with isolated critical points at the origin $0\in
\mathbb{R}^{m}$, obtained by projecting the map germ $\psi$ in the
family $L_{-\theta}$ of all lines through the origin in the plane
$\mathbb R^{2}.$

\end{abstract}

\maketitle


\section{Introduction}

In \cite{mi} Milnor proved that if \[\psi : \left( {\mathbb
C}^{n+1},0\right )\longrightarrow ({\mathbb C},0)\,, \] is the
germ of a holomorphic function with a critical point at $0$, then
for every sufficiently small $\epsilon > 0$ the map
$\displaystyle{\frac{\psi}{||\psi||}}:{S}_{\epsilon}^{2n+1}\setminus
{K}_{\epsilon} \to {S}^{1}$ is the projection map of a smooth
locally trivial fibre bundle, where
${K}_{\epsilon}={\psi}^{-1}(0)\cap{S}_{\epsilon}^{2n+1}$ is the
link of singularity at $0$. This is the Milnor fibration for 
holomorphic singularities functions germs.

\hspace{1cm}

Milnor also proved in the last chapter of his book a fibration
theorem for real singularities. He showed that if

$$\psi :(\mathbb{R}^m ,0)\to (\mathbb{R}^p ,0)\,,
m\geq p\geq 2,$$

\noindent is a real analytic map germ whose derivative $D{\psi}$
has rank $p$ on a punctured neighborhood of $0\in \mathbb R^m$,
then there exists $\epsilon >0$ and $\eta >0$ sufficiently small
with $0<\eta \ll\epsilon <1$, such that considering
$E:=B_{\epsilon}^{m}(0)\cap \psi ^{-1}( S_{\eta}^{p-1})$,
$B_{\epsilon}^{m}(0)$ the open ball centered in $0\in
\mathbb{R}^{m}$ and radius $\epsilon $, we have that $\psi
_{\mid_{E}}: E \to S_{\eta}^{p-1}$ is a smooth locally trivial
fibre bundle. Milnor also proved the existence of a diffeomorphism
that pushes $E$ to $S^{m-1}_{\epsilon}\setminus N_{K_{\epsilon}}$,
where $N_{K_{\epsilon}}$ denotes a tubular neighborhood of the
link $K_{\epsilon}$ in $S^{m-1}_{\epsilon}.$ Moreover, this fibre
bundle can be extended to the complement of the link in the sphere
$S^{m-1}_{\epsilon}\setminus K_{\epsilon}$, with each fiber being
the interior of a compact manifold bounded by $K_{\epsilon}$. But,
in all these constructions we cannot guarantee that the map
$\displaystyle{\frac{\psi}{\|\psi \|}}$ is the projection
of the fibration, as it is easily shown by the example below.

\begin{example}\cite[page 99]{mi}
\begin{equation*}
\begin{cases}
P=x \\
Q=x^{2}+y(x^{2}+y^{2})
\end{cases}
\end{equation*}
\end{example}

\begin{definition}\cite{rsv}

Let $\psi:(\mathbb{R}^{m},0)\to (\mathbb{R}^{p},0)$, $m\geq p\geq
2$, be a map germ with isolated singularity at the origin. If for
all $\epsilon >0$ sufficiently small, the map $\phi
=\displaystyle{\frac{\psi}{\|\psi\|}}:S^{m-1}_{\epsilon}\setminus
K_{\epsilon} \,\to \,{S}^{p-1}\,$, is a projection of a smooth
locally trivial fibre bundle, where $K_{\epsilon}$ is the link of
singularity at $0$, we say that the map germ  satisfies the {\bf
Strong Milnor condition} at $0\in \mathbb{R}^{m}.$
\end{definition}

The problem of studying real isolated singularities for which the
map $\displaystyle{\frac{\psi}{\|\psi \|}}$ extends
as a smooth projection of the fibre bundle
$S^{m-1}_{\epsilon}\setminus {K}_{\epsilon} \,\to \, S^{p-1}$, as in
the holomorphic case, was first studied by A. Jacquemard in \cite
{ja}, \cite{ja1}, by J. Seade, Ruas and Verjovsky in \cite{rsv},
and by the author and Ruas in \cite{rs}.

The Jacquemard's approach \cite {ja} was the following: considering
$\psi =(P,Q):(\mathbb R^m ,0)\to (\mathbb R^2 ,0)\ $ an analytic
real map germ with isolated singularity at $0\in \mathbb{R}^m,$ he
gave two conditions which were sufficient to guarantee that the
map $\displaystyle{\frac{\psi}{\|\psi \|}}$ extends
to all $S^{m-1}_{\var}\setminus {K}_{\epsilon}$ as a smooth
projection map of a locally trivial fiber bundle over $S^{1}$,
i.e, the function germ $\psi$ satisfy the Strong Milnor condition at
origin. The first condition (A) is geometric: the angle between
the gradient vector fields $\nabla P$ and $\nabla Q$ has an upper bound
smaller than 1; the second condition (B) is algebraic: the
Jacobian ideals of $P$ and $Q$ have the same integral closure in
the local ring of real analytic function germs at $0 \in \mathbb
R^m$. With these tools the author recovered some main ideas given
by Milnor on his book \cite{mi} to construct the locally trivial
fibre bundle.

In another direction, using stratification theory and singularity
theory, in \cite{rs} the authors proved that the Jacquemard's conditions
are not necessary for the existence of a Milnor fibration. The result were
the following:

still considering $\psi =(P,Q):(\mathbb R^m ,0)\to (\mathbb R^2 ,0)\ $ a
real analytic map germ, with isolated singularity at $0\in
\mathbb{R}^m$, and $\Psi(x,t)$ a convenient family of functions
associated to map $\psi $ (called Seade's family, see
Definition~2.4, page 5), $X=\Psi^{-1}(0)\setminus{\{0\}\times
\mathbb{R}}$ and $Y=\{0\}\times \mathbb{R}$ a stratification of
analytic variety $\Psi^{-1}(0) $, holds:

\begin{theorem}\cite{rs}
If the real analytic map germ $\psi=(P,Q):(\mathbb{R}^{m},0)$ $\to
(\mathbb{R}^{2},0)$ satisfies the Jacquemard hypotheses, then the
pair $(X,Y)$ as above satisfies the Verdier's $(w)-$condition.
\end{theorem}

It is well known that, in real and complex subanalytic
settings we have these sequences of implications: Verdier's
$(w)-$condition $\Rightarrow$ Kuo's ratio test $\Rightarrow$
$(b)-$ Whitney condition $\Rightarrow$ Bekka's $(c)-$ condition.

\begin{theorem} \label{T1} \cite{rs}  If the pair $(X,Y)$, as above, satisfy the Bekka's
$(c)-$ condition on $\mathbb{R}^{m}\times \mathbb{R}$ with respect the control function
$\rho(x,\theta)=\Sigma_{i=1}^{n}x^{2}_{i}$(or, in another words,
the pair $(X,Y)$ satisfies the $A_{\rho}$-Thom condition), then
the map $ \displaystyle{\frac{\psi}{\|\psi \|}}$
extends as a smooth projection of locally trivial fibre bundle
$S^{m-1}_{\epsilon}\setminus {K}_{\epsilon} \,\to \, S^{1}$.

\end{theorem}

In the example below, is easy to see that, the pair $(X,Y)$
associated to Seade's family satisfies the $A_{\rho}-$Thom
condition for $\rho(x,y,\theta)=x^{2}+y^{2}$ and $\rho_{\mid_{X}}$
is a submersion, i.e. the pair $(X,Y)$ satisfies $(c)-$regularity
condition (see definition \ref{DefBK}), but does not satisfy the
Jacquemard hypothesis((B), in this case), showing that the
Jacquemard hypotheses is stronger than the hypothesis given by the
authors in Theorem \ref{T1}.

\begin{example}
\begin{equation*}
\begin{cases}
P=xy \\
Q=x^{2}-y^{4}
\end{cases}
\end{equation*}
\end{example}

\hspace{1cm}

In this work, still using the family
$\displaystyle{\Psi(x,\theta)}$, we give another point of view to
get the Strong Milnor fibration. Actually, following the approach
given by Milnor in \cite{mi} and by A. Jacquemard in \cite{ja}, we
des-cribe a condition weaker than $(c)-$regularity, as given in
\cite{rs}. The main Theorem is:

\begin{theorem}\label{mt}
Let $\psi =(P,Q):(\mathbb R^m ,0) \to (\mathbb R^2 ,0)\ $ be a
real analytic map germ, with isolated critical point at $0\in
\mathbb{R}^{m}$, and $\Psi(x,\theta)$ the associated Seade's
family for the map $\psi $. Suppose that for all
$x\in U\setminus \{0\}$, where $U$ is an open domain of $\psi$, we have:\\
$\displaystyle{\mid\langle\frac{\nabla_{x}\Psi_{\theta}(x)}{\|\nabla_{x}\Psi_{\theta}(x)\|},
\frac{x}{\|x\|} \rangle\mid \leq 1-\rho }$; $0<\rho \leq 1$;
$\forall \theta\in \mathbb{R}$. Then, $\psi$ satisfies the Strong
Milnor condition at $0\in \mathbb{R}^{m},$ i.e. there exist
$\epsilon_{0}>0$, sufficiently small, such that $\forall
\epsilon$, $0<\epsilon \leq \epsilon_{0}$, the projection map
$\displaystyle{\frac{\psi}{\|\psi\|}:{S_{\epsilon}^{m-1}}\setminus
K_{\epsilon} \to {S^{1}}}$ is a smooth locally
trivial fibre bundle.
\end{theorem}

\section{Definition and Basic results}

We briefly recall some definitions and basic results. For more
details see \cite{kbsk}, \cite{kb1}.

Let $M$ be a smooth Riemannian manifold, $X$ and $Y$
submanifolds of $M$, such that $Y\subset \overline{X}$. Let
$(T_Y,\pi ,\rho )$ be a tubular neighbourhood of $Y$ in $M$
together with a projection $\pi : T_Y \to Y,$ associated to a
smooth non-negative control function $\rho$ with $\rho ^{-1}(0)=Y$
and $\nabla \rho(x)\in \ker(d \pi (x))$.

\begin{definition}\cite{kbsk}:

The pair $(X,Y)$ satisfies condition (m) if there exists a real
number $\epsilon> 0$ such that

\begin{align*}
(\pi ,\rho )\mid _{X \cap T^{\epsilon}_{Y}}:&X \cap T^{\epsilon}_{Y} \to Y \times \mathbb{R} \\
&x\longmapsto (\pi (x), \rho (x))
\end{align*}

\noindent is a submersion, where $T^{\epsilon}_{Y}:=\{x\in T_{Y}/
\rho (x)< \epsilon \}$.

\end{definition}

Geometrically, the condition (m) says that the submanifold $X$ is
transverse to level $\rho =c$ inside the open tubular neighborhood
$T^{\epsilon}_{Y}$.

\begin{definition} (Bekka's condition) \cite{kb2} \label{DefBK}

We say that a pair of strata $(X,Y)$ satisfies $(c)-$ regularity
condition with respect a smooth non-negative control function
$\rho:M\to \mathbb{R}^{+}$, if the following holds:

\begin{itemize}

\item [i)] $\rho^{-1}(0)=Y;$

\item [ii)] $\rho \mid_{X}$ is a submersion;

\item [iii)] Let $\{x_{i} \}$ a sequence of points in $X$ such
that $\{ x_{i} \} \rightarrow y \in Y$ and
$\ker(d_{x_{i}}\rho\mid_{X})\rightarrow \tau ;$ then
$T_{y}Y\subset \tau .$

\end{itemize}

Considering $Star(Y)=\{X: X~is ~statum ~ such
~that~\overline{Y}\subset X \}$, in a general way, the property
$ii)$ of definition says that $\rho\mid_{Star(Y)}$ is a stratified
map and $iii)$ says that $\rho\mid_{Star(Y)}$ is a Thom map. The
following proposition is indeed a geometric easy way to see the item
$iii)$.

\newpage

\begin{proposition}\cite{kb2}
The property $iii)$ of the above definition is equivalent to

\begin{center}

$\displaystyle{\lim_{x\rightarrow y}
\Pi_{Y}(\frac{grad_{x}(\rho\mid _{X})}{\| grad_{x}(\rho\mid _{X})
\|})=0}$

\end{center}

where $\Pi_{Y}$ is the orthogonal projection on $T_{y}Y$and ${\bf
0}\in T_{y}Y.$

\end{proposition}

Now let $F:\mathbb{R}^{m}\times \mathbb{R},0\times \mathbb{R}\to
\mathbb{R},0$ be a one-parameter family of function-germs,
$F(x,\theta)=F_{\theta}(x)$, $X:=F^{-1}(0)\setminus (0\times
\mathbb{R})\subset \mathbb{R}^{m}\times \mathbb{R}$, $Y:=0\times
\mathbb{R}$ and $X_{\theta}=F_{\theta}^{-1}(0)\subset
\mathbb{R}^{m},0$.

We say that the family $F(x,\theta)$ has $\rho-$Milnor's radius
uniformly, if there exist $\epsilon_{0}>0$ such that the pair
$(X,Y)$, defined as above, satisfies condition (m), with respect
some control function $\rho$.
\end{definition}

\begin{remark}

\item[ 1.] If the definition holds for some $\epsilon_{0}>0$, it
also holds for all $\epsilon $, $0<\epsilon \leq \epsilon_{0};$

\item[ 2.] Using the control function $\rho
(x_{1},...,x_{m},\theta)= \Sigma_{i=1}^{m}x_{i}^{2}$ in
$\mathbb{R}^{m+1}$, this definition says that the strata $X$ is
transverse to all Euclidean cylinder into the tubular neighborhood
$T^{\epsilon_{0}}_{Y}$. More precisely, the manifolds
$X_{\theta}=F_{\theta}^{-1}(0)$ are transverse to spheres
$S_{\epsilon}^{m-1}$, for all $0<\epsilon \leq \epsilon _{0}$.

\end{remark}

Finally we recall Seade's method given in \cite{s2}, \cite{rs}.
Consider a real analytic map germ $\psi :\mathbb {R}^m,0 \to
\mathbb {R}^{2},0$ and identify $\mathbb{R}^{2}$ with
$\mathbb{C}$, we have $\psi(x)=(P(x),Q(x)) \approx P(x)+iQ(x)$,
where $i^{2}=-1$. Let $\pi_{\theta}: \mathbb{C} \to {L}_{\theta}$
be the orthogonal projection to the line ${L}_{\theta}$ through
the origin, forming angle $\theta$ with the horizontal axis in
$\mathbb C$ and take the composition $\Psi(x,\theta)= \pi_{\theta}
\circ\psi (x)$.

\begin{lemma}\label{L3}\cite{s2}
 Let $U\subseteq \mathbb{R}^{m}$ be a neighborhood
 of $0$ such that for every $x\in U\setminus \{0\}$, $\psi$ has maximal rank at $x.$
 Then the following hold:

 \begin{itemize}

 \item [(i)] $U=\cup _{\theta} (M_{\theta}\cap U)$, $0\leq \theta <\pi $ .

 \item [(ii)] $M=\cap_{\theta} M_{\theta} = M_{\theta _{1}} \cap M_{\theta _{2}}
  $, where $M=\psi ^{-1}(0)$, $\theta _{1}\neq \theta _{2},\theta _{1}, \theta _{2}
  \in [0,\pi)$.

 \item [(iii)] For each $\theta \in [0,\pi),$
  $M_{\theta}=E_{\theta}\cup M \cup E_{\theta + \pi}$,
  where $E_{\alpha }=\widetilde{\phi}^{-1}(e^{i\alpha})$
  and $M=\psi^{-1}(0)$, with $\widetilde{\phi}:U\setminus M \to
{S}^{1}$, $\displaystyle{\widetilde{\phi }(x)=i\frac{\overline
{\psi (x)}}{\|\psi (x)\|}}$.

 \item [(iv)] For each $\theta \in
  [0,\pi)$, $M_{\theta}^{*}=M_{\theta}\setminus \{0\}$ is a real smooth submanifold
   of real codimension $1$ of  $U\setminus \{0\}$, given
  by the union  of
  $E_{\theta},$ $E_{\theta + \frac{\pi}{2}}$ and $M\setminus \{0\}$.
 \end{itemize}
 \end{lemma}

\begin{definition}

The family $\Psi :(\mathbb{R}^{m}\times \mathbb{R},0)\to
(\mathbb{R},0)$, $\Psi(x,\theta)=\pi_{\theta}\circ \psi(x)$,
defined as above, will be called the Seade's family associated to the map
germ $\psi =(P,Q)$.

\end{definition}

\section{Tools}

In what follows let $\psi=(P,Q):\mathbb{R}^{m},0\to
\mathbb{R}^{2},0$ be a real analytic map germ with isolated
critical point at $0\in \mathbb{R}^{m}$, $U \ni 0$ some open
domain of $\displaystyle{\psi}$ with the decomposition given by
Lemma \ref{L3} and $\Psi:\mathbb{R}^{m}\times \mathbb{R},0 \to
\mathbb{R},0$ the associated Seade's family to the real analytic
map germ $\psi$, satisfying the hypothesis of Theorem\ref{mt}.

\vspace{0.2cm}

In this section we prove two preliminary lemmas. In the first
we show that the projection $\displaystyle{\frac{\psi}{\|\psi \|}}$
does not have critical point in $S_{\epsilon}^{m-1}\setminus
K_{\epsilon}$ for all $\epsilon $ sufficiently small. Using the hypothesis of Theorem\ref{mt},
it will be possible to guarantee the existence of $\epsilon_{0}>0
$ sufficiently small, such that the manifolds $M_{\theta}\setminus
\{0\}$ are transverse to $S_{\epsilon}^{m-1}$ for all $\epsilon,
0<\epsilon \leq \epsilon_{0}$, and for all $\theta $.

\begin{lemma}\label{L1}

There exists $\epsilon_{0}>0$, sufficiently small, such that for
all $ x\in B_{\epsilon_{0}}^{m}(0)\setminus
\{\psi^{-1}(0)\}\subset U,$ the vector fields
$\nabla_{x}\Psi_{\theta}(x)$ and $\gamma (x):=P(x)\nabla
Q(x)-Q(x)\nabla P(x)$ are parallel.

\end{lemma}

\begin{proof} Let $B_{\epsilon_{0}}^{m}(0)=\cup [M_{\theta}\cap
B_{\epsilon_{0}}^{m}(0)]$. For each $x\in
B_{\epsilon_{0}}^{m}(0)\setminus \{\psi^{-1}(0)\}$, $\exists
\theta \in \mathbb{R}$ such that
$\cos(\theta)P(x)=\sin(\theta)Q(x)$, since $x$ belongs to some $
M_{\theta}$.

\vspace{0.2cm}

\noindent So, consider the following cases:

\vspace{0.2cm}

1) $\sin(\theta)Q(x)\neq 0$;

2) $\sin(\theta)Q(x)= 0$.

\vspace{0.2cm}

\noindent If $1)$ holds, we have

$\displaystyle{P(x)=\frac{\sin(\theta)}{\cos(\theta)}Q(x)}$ and
$\gamma(x)=\displaystyle{\frac{\sin(\theta)}{\cos(\theta)}Q(x)\nabla
Q(x)-}$  $\displaystyle{Q(x)\nabla P(x)}$
$=\displaystyle{-\frac{Q(x)}{\cos(\theta)}}(\cos(\theta)\nabla
P(x)-\sin(\theta)\nabla Q(x))$. Then, $\gamma
(x)=-\displaystyle{\frac{Q(x)}{\cos(\theta)}\nabla
\Psi_{\theta}(x)}$.

\vspace{0.2cm}

\noindent If we have the particular case 2) $\sin(\theta)Q(x)=0$, consider
again two situations:

\vspace{0.2cm}

\noindent i) If $\sin(\theta)=0$ then $\cos(\theta)\neq
0$ and $P(x)=0$. Since $x\in
B_{\epsilon_{0}}^{m}(0)\setminus \psi^{-1}(0)$, then $\gamma(x)=-Q(x)\nabla P(x)$ and
$\nabla_{x}\Psi_{\theta}(x)=\cos(\theta)\nabla P(x)$.

\vspace{0.2cm}

\noindent ii) If $Q(x)=0$ then $\cos(\theta )=0$ and $P(x)\neq 0$, since $x\in
B_{\epsilon_{0}}^{m}(0)\setminus \psi^{-1}(0)$. Then,
$\sin(\theta )\neq 0$,
$\gamma (x)=P(x)\nabla Q(x)$ and $\nabla_{x}\Psi_{\theta
}(x)=-\sin(\theta )\nabla Q(x)$.

\vspace{0.5cm}

\noindent Therefore, in both cases we have $ \gamma (x)$ is parallel to
$\nabla_{x}\Psi_{\theta }(x).$

\end{proof}

\noindent In \cite{ja} the author proved that the critical points
of the map
$\displaystyle{\frac{\psi}{\|\psi\|}}:S_{\epsilon}^{m-1}\setminus
K_{\epsilon}\to S^{1}$ are precisely the points $x\in
S_{\epsilon}^{m-1}$ such that the vector fields
$\gamma(x)=P(x)\nabla Q(x)-Q(x)\nabla P(x)$ and $x$ are parallel.
However, the hypothesis of the main theorem implies that the
vector fields $x$ and $\nabla_{x}\Psi_{\theta}(x)$ are
transversal. So, using Lemma \ref{L1} above we have that the
projection
$\displaystyle{\frac{\psi}{\|\psi\|}}:S_{\epsilon}^{m-1}\setminus
K_{\epsilon}\to S^{1}$ is a submersion, for all $\epsilon >0$
sufficiently small.

\vspace{0.3cm}

In the following result we construct a smooth vector fields in $S_{\epsilon}^{m-1}\setminus
K_{\epsilon}$, whose solution do not fall in the empty $K_{\epsilon}$, for finite time. This
result will guarantee that the projection $\displaystyle{\frac{\psi}{\|\psi\|}}$ is a
onto submersion.

\begin{lemma}\label{L2}

There exists a smooth vector fields $\omega $ tangent to
$S_{\epsilon}^{m-1}\setminus K_{\epsilon}$, such that:

\item[(i)] $\displaystyle{\langle \omega(x),\frac{ P(x)\nabla
Q(x)-Q(x)\nabla P(x)}{P^{2}(x)+Q^{2}(x)}\rangle=1},$

\item[(ii)] $\mid \displaystyle{\langle \omega(x),\frac{
P(x)\nabla P(x)+Q(x)\nabla Q(x)}{P^{2}(x)+Q^{2}(x)}\rangle \mid
\leq M}; $ $\displaystyle{M>0}$, for each $\epsilon >0$,
sufficiently small.
\end{lemma}

\proof For each $x\in S_{\epsilon}^{m-1}\setminus K_{\epsilon }$
define $\displaystyle{u(x):=\gamma (x)- \langle \gamma(x),
\frac{x}{\|x \|}}\rangle.\frac{x}{\|x\|}$, the projection of
$\gamma(x)$ in $T_{x}(S_{\epsilon}^{m-1}\setminus K_{\epsilon })$.
Under the hypotheses of the main Theorem and Lemma \ref{L1} above,
this vector field is smooth and never zero.

\noindent Let
$w(x):=\displaystyle{(\frac{P^{2}(x)+Q^{2}(x)}{\|u(x)\|}
).\frac{u(x)}{\|u(x)\|}}$. So we have:

\vspace{0.3cm}

i) $\displaystyle {\langle w(x),\frac{P(x)\nabla Q(x)-Q(x)\nabla
P(x)}{P(x)^{2}+Q(x)^{2}} \rangle }$ = $\displaystyle{\langle
\frac{u(x)}{\|u(x)\|^{2}}, \gamma(x) \rangle }$=

$\displaystyle{\frac{1}{\|u(x)\|^{2}}\langle u(x), \gamma (x)-
\langle \gamma (x), \frac{x}{\|x\|}} \rangle
.\frac{x}{\|x\|}\rangle$ =
$\displaystyle{\frac{1}{\|u(x)\|^{2}}}\langle u(x),u(x) \rangle
=1$,

\vspace{0.1cm}

\noindent where the second equality we use the fact
$\displaystyle{ u(x)\perp \langle \gamma(x),\frac{x}{\|x\|}
\rangle . \frac{x}{\|x\|}}$. This proves the first statement.

\vspace{0.2cm}

ii) $\displaystyle {\langle w(x),\frac{P(x)\nabla P(x)+Q(x)\nabla
Q(x)}{P(x)^{2}+Q(x)^{2}}\rangle }$ =$ \displaystyle  {\frac{1}{2}
\langle w(x),\frac {\nabla (\|\psi(x) \|^{2})}{P(x)^{2}+Q(x)^{2}}
\rangle }$= \\ $\displaystyle{\frac{1}{2}\langle
\frac{u(x)}{\|u(x)\|^{2}}, \nabla (\|\psi(x)\|^{2})\rangle }$
$\leq \displaystyle{\frac{1}{2}\frac{\|u(x)\|.\|\nabla
(\|\psi(x)\|^{2})\|}{\|u(x)\|^{2}}}$
$=\displaystyle{\frac{1}{2}\frac{\|\nabla(\|\psi(x)\|^{2})\|}{\|u(x)\|}}.$

\vspace{0.3cm}

\noindent Since
$\displaystyle{\|u(x)\|^{2}=\|\gamma(x)\|^{2}-\langle
\gamma(x),\frac{x}{\|x\|} \rangle ^{2}}$ $=
\|\gamma(x)\|^{2}(1-\langle \frac{\gamma(x)}{\|\gamma(x)\|},
\frac{x}{\|x\|} \rangle ^{2})>$\\

$> \rho ^{2}\|\gamma(x)\|^{2}$ then,\\
$\displaystyle{\mid \langle w(x), \frac{\nabla
(\|\psi(x)\|^{2})}{P(x)^{2}+Q(x)^{2}} \rangle \mid }$ $\leq
\displaystyle{\frac{1}{2}\frac{\|\nabla(\|\psi(x)\|^{2})\|}{\|u(x)\|}}$
$<\displaystyle{\frac{1}{2\rho }\frac{\|\nabla
(\|\psi(x)\|^{2})\|}{\|\gamma (x)\|}}.$


\noindent It is enough to verify that for all $x\in
{S_{\epsilon}}^{m-1}\setminus K_{\epsilon}$, it is possible to get
upper bounds for the last expression by using the curve selection
lemma.

For this, consider $\delta >0$ a real number sufficiently small,
and a non-constant real analytic curve $\alpha : [0,\delta )\to
S_{\epsilon}^{m-1}$, with $\alpha(t)\in
S_{\epsilon}^{m-1}\setminus K_{\epsilon}$ and $\alpha (0)\in
K_{\epsilon}$. Using the Taylor expansion we have:

$\begin{cases}

\alpha(s)=\alpha_{0} + \alpha_{1}s^{n}+...;n\geq 1, e~
\alpha_{0}\in K_{\epsilon}.\\

P(\alpha(s))=P_{0}+P_{1}s^{r}+...;r\geq 1;\\

\nabla P(\alpha (s))=a_{0}+a_{1}s^{l}+...,l\geq 1,~a_{0}\neq 0;\\

Q(\alpha (s))=Q_{0} +Q_{1}s^{k}+...., k\geq 1;\\

\nabla Q(\alpha (s))= b_{0}+b_{1}s^{p}+...,p\geq 1, ~b_{0}\neq
0.\\

\end{cases}$

\vspace{0.2cm}

\noindent Since $\alpha_{0}\in K_{\epsilon}=S_{\epsilon}^{m-1}\cap
\psi^{-1}(0)\Longrightarrow $ $P(\alpha(0))=P_{0}=0$ and
$Q(\alpha(0))=Q_{0}=0$. Therefore

$\begin{cases}
P(\alpha(s))=P_{1}s^{r}+..., r\geq 1, P_{1}\neq 0 \\
Q(\alpha(s))=Q_{1}s^{k}+..., k\geq 1, Q_{1}\neq 0
\end{cases}$

\vspace{0.2cm}

\noindent Since $\Sigma (P,Q)=\{0\}$ the vectors $\nabla P(\alpha
(0))$ and $ \nabla Q(\alpha (0))$ are not parallel outside of
origin.

Then,
$\displaystyle{\frac{\|\nabla(\|\psi(\alpha(s))\|^{2})\|^{2}}{\|\gamma
(\alpha(s))\|^{2}}}=$\\ $\displaystyle{
\frac{\|(P_{1}s^{r}+...)(a_{0}+a_{1}s^{l}+...)+(Q_{1}s^{k}+...)(b_{0}+b_{1}s^{p}+
...)\|^{2}}{\|(P_{1}s^{r}+...)(b_{0}+b_{1}s^{p}+...)-(Q_{1}s^{k}+...)(a_{0}+a_{1}s^{l}+
...)\|^{2}}=}$\\

\noindent
$\displaystyle{\frac{\|P_{1}a_{0}s^{r}+Q_{1}b_{0}s^{p}+...\|^{2}}{\|P_{1}b_{0}s^{r}-Q_{1}a_{0}s^{p}+
...\|^{2}}}=$
$\displaystyle{\frac{\|P_{1}a_{0}s^{r}+Q_{1}b_{0}s^{p}\|^{2}+
s^{2}(U(s))}{\|P_{1}b_{0}s^{r}-Q_{1}a_{0}s^{p}\|^{2}+
s^{2}(V(s))}}$; for some analytic functions $U(s)$, $V(s)$ in
variable $s$.

\vspace{0.2cm}

\noindent The last expression has a upper bound for s small
enough, if we take any natural $r,p$ with $r\neq p$. Now it
remains to check if $r=p$ and $P_{1}b_{0}-Q_{1}a_{0}=0$ because,
in this particular case, the order of denominator can be bigger
than the order of numerator when $s$ goes to zero. It means that
the last expression above goes to infinity.

But $\displaystyle{P_{1}b_{0}-Q_{1}a_{0}=0 \Longleftrightarrow
b_{0}=\frac{Q_{1}}{P_{1}}a_{0}}$ and

$\begin{cases}
\nabla P(\alpha(s))=a_{0}+a_{1}s^{l}+... \\
\nabla Q(\alpha (s))=\frac{Q_{1}}{P_{1}}a_{0}+b_{1}s^{k}+...
\end{cases}$

\vspace{0.2cm}

Then, $\nabla P(\alpha(0))//\nabla Q(\alpha(0))$ in
$S_{\epsilon}^{m-1}$, contradicting  $\sum (P,Q)=\{0\}$.

\endproof

\section{Main result}

\begin{proposition}\label{P1}

Let $\psi:(\mathbb{R}^{m},0)\to (\mathbb{R}^{2},0)$ with isolated critical point
at $0\in \mathbb{R}^{m}$, such that the associated Seade's family satisfies
the hypothesis of main theorem, then there exist $\epsilon_{0}$ such that, for all
$0<\epsilon \leq \epsilon_{0}$, the map $\displaystyle{\frac{\psi }{\|\psi
\|}:S_{\epsilon}^{m-1}\setminus K_{\epsilon}\to S^{1}}$ is a
smooth projection of a locally trivial fibre bundle.
\end{proposition}

\proof

\noindent Now we are ready to construct the prove of Theorem 1.6.

Under it hypothesis and by Lemma 3.2, there exists a smooth
vector fields $\omega $ tangent to
$S_{\epsilon}^{m-1}\setminus K_{\epsilon}$, satisfying conditions
i) and ii) of lemma. Now taking the flow $\phi_{t}(x) $ of
$\omega \in \displaystyle {T_{x}(S_{\epsilon}^{m-1}\setminus
K_{\epsilon})},$ with $\phi_{0}(x)=x$, and considering
$\displaystyle{\frac{\psi(x)}{\|\psi(x)\|}=e^{i\theta (x)}}$, we have
$\theta (x)=Re(-i\ln \psi(x))$ and

\noindent $\displaystyle {\frac {d}{dt}(\theta
(\phi_{t}(x)))}=\langle {\dot{\phi }_{t}(x)},\frac
{P(\phi_{t}(x))\nabla Q(\phi_{t}(x))-Q(\phi_{t}(x))\nabla
P(\phi_{t}(x))}{P^{2}(\phi_{t}(x))+Q^{2}(\phi_{t}(x))} \rangle
=1$\\

\noindent then, $\theta (\phi_{t}(x))=t+c;$ $\theta
(\phi_{t}(x))=t+\theta_{0}$, where $\theta (\phi_{0}(x))=\theta
(x)=\theta _{0}.$ it the initial value.

\noindent Therefore, if the flow $\phi_{t}$ is well defined
$\forall t\in \mathbb{R}$, we have $\displaystyle{\frac{\psi}{\|
\psi \|}:S_{\epsilon}^{m-1}\setminus K_{\epsilon}}$
$\displaystyle{\to S^{1}}$ is an onto smooth submersion over
$S^{1}$, because the projection function
$\displaystyle{\frac{\psi}{\|\psi \|}}$ wrap $\phi_{t}(x)$ around
$S^{1}$, all time $t$.

\noindent If for some fixed $t_{0}\in \mathbb{R}$, $t\to t_{0}
\Longrightarrow$ $\phi_{t}\to K_{\epsilon}$, then $ \|
\psi(\phi_{t})\| \to 0$ and $\log \| \psi(\phi_{t})\| \to \infty
$. \\But, $\mid \frac{d}{dt}(\log ( \| \psi(\phi_{t})\|^{2})) \mid
$ = $\mid \langle \dot{\phi_{t}},
\displaystyle{\frac{\nabla (\|\psi(\phi_{t})
\|^{2})}{P(\phi_{t})^{2}+Q(\phi_{t})^{2}}}\rangle \mid \leq M$ has
bounded derivative, so this flow is well defined for all $t\in
\mathbb{R}$.

\noindent Now using the same idea of \cite[page 43]{mi}, see also
\cite[page 22]{ja1}, consider $\displaystyle{\pi
=\frac{\psi}{\|\psi\|}}$, $x_{0}\in S_{\epsilon}^{m-1}\setminus
K_{\epsilon}$ fixed, define $h_{x_{0}}(t)=\phi_{x_{0}}(t)$. For
each $t$, it is well known that $h_{t}:
S_{\epsilon}^{m-1}\setminus K_{\epsilon}\to
S_{\epsilon}^{m-1}\setminus K_{\epsilon}$ is a
$C^{\infty}-$diffeomorphism given by the flow of $\omega$ in Lemma
\ref{L2}, and if $e^{is}\in S^{1}$,
$h_{t}(\pi^{-1}(e^{is}))=\pi^{-1}(e^{i(s+t)})$, this says that
this flow is transverse to all fiber $F_{t}:=\pi^{-1}(t)$.
Furthermore, if we consider $U_{\alpha}$ a neighborhood of $e^{i
\alpha}$ in $S^{1}$, small enough, we have the following commutative
diagram:


\begin{center}
$\xymatrix{U_{\alpha}\times \displaystyle{\pi^{-1}(\alpha)}
\ar[d]_{\pi_{1}} \ar[r]^{h_{t}}
& \displaystyle{\pi^{-1}(U_{\alpha})} \ar[ld]^{\pi} \\
\displaystyle{U_{\alpha}} & \\}$
\end{center}


\noindent where $\pi_{1}$ is the projection on first coordinate.

\endproof



\section{Comparing this result with $(c)-$regularity condition}

It is easy to see that the hypothesis of Theorem \ref{mt} implies
$(m)-$ condition for the pair $(X,Y)$, in a tubular
neighborhood $T_{Y}^{\epsilon}=\{ x\in \mathbb{R}^{m}\times
\mathbb{R}: \rho(x,\theta)=\sum_{i=1}^{m}x^{2}_{i}< \epsilon \}$,
where $X=\Psi^{-1}(0)\setminus{\{0\} \times \mathbb{R}}$ and
$Y=\{0\}\times \mathbb{R}$.

\begin{corollary}

If the pair $(X,Y)$, defined as in Theorem \ref{T1}, satisfy the
Bekka's $(c)-$ condition on $\mathbb{R}^{m}\times \mathbb{R}$ with
respect the control function
$\rho(x,\theta)=\Sigma_{i=1}^{n}x^{2}_{i}$, then
$\displaystyle{\psi }$ satisfy the strong Milnor condition at
origin.

\end{corollary}

\proof It turn out that Bekka's $(c)-$ regularity implies
condition $(m)$ and Whitney $(a)-$regularity condition for the
pair of strata $(X,Y)$\cite{kbsk}.

\endproof

\vspace{0.3cm}

The example below shows a real analytic maps germ which satisfies
the hypothesis of Theorem \ref{mt} but the associated pair of strata
$(X,Y)$ does not satisfy Whitney $(a)-$ regularity.
More details can be found in \cite{acs}:

\begin{example}
\begin{equation*}
\begin{cases}
P=x \\
Q=yx^{2}+y^{3}
\end{cases}
\end{equation*}
\end{example}

Observe that
$\Psi(x,y,\theta)=\cos(\theta)x-\sin(\theta)(yx^{2}+y^{3})$, so
$\nabla \Psi_{\theta}(x,y)=(\cos(\theta)-2xy\sin(\theta),
-(x^{2}+3y^{2})\sin(\theta)).$ Let's verify that in some punctured
neighborhood of the origin the vectors $\nabla \Psi_{\theta}(x,y)$ and
the position vector $(x,y)$ are not parallel, for all $\theta \in
\mathbb{R}$. It will be enough to solve the following system:

$\begin{cases}
\left\langle \nabla_{(x,y)} \Psi_{\theta}(x,y),(-y,x) \right\rangle =0\\
\Psi_{\theta}(x,y)=0
\end{cases}$

\vspace{0.2cm}

Or,

$\begin{cases}
y\cos(\theta)+(xy^{2}+x^{3})\sin(\theta) =0\\
x\cos(\theta)-(yx^{2}+y^{3})\sin(\theta)=0
\end{cases}$

\vspace{0.2cm}

The matrix of the system is

$$\begin{pmatrix}
y & xy^{2}+x^{3}\\
x & -(yx^{2}+y^{3})

\end{pmatrix}.
\begin{pmatrix}

\cos(\theta) \\
\sin(\theta)

\end{pmatrix}=
\begin{pmatrix}
0 \\
0
\end{pmatrix}$$

The vector $(\cos(\theta),\sin(\theta))$ is always never zero, so
we have to check the rank of the matrix on the right, but this
determinant is
$-y(yx^{2}+y^{3})-x(xy^{2}+x^{3})=-(x^{2}+y^{2})^{2}$. It is equal
zero, iff $x=y=0$. So, this map germ satisfies the hypothesis of
main theorem.

In order to see that the associated pair $(X,Y)$ of the map germ
above does not satisfy the Whitney $(a)-$ regularity, consider the
sequences $x_{i}, y_{i}, \theta_{i}$, with $x_{i}\to 0, y_{i}=0,
\theta_{i}=\frac{\pi}{2}$. It is easy to see that Whitney
$(a)-$regularity fails along this sequence.

It means that the hypothesis given in Theorem \ref{mt} is weaker
than the $(c)-$ regularity condition over the pair $(X,Y)$ given
by the authors in \cite{rs}.

\vspace{0.3cm}

\begin{example}
\begin{equation*}
\begin{cases}
P=z(x^{2}+y^{2}+z^{2}) \\
Q=y-x^{3}
\end{cases}
\end{equation*}
\end{example}

It is easy to see that this map germ has an isolated critical point at
origin and the link is given by $K_{\epsilon}=\{z=0, y=x^{3}\}\cap
S^{2}_{\epsilon}$, i.e, two points. The points where
$\displaystyle{\nabla_{(x,y,z)}\Psi_{\theta}(x,y,z)}$ and the vector position $(x,y,z)$ are parallels 
satisfies the following system, for some $\lambda \in \mathbb{R}^{*}$:

\vspace{0.3cm}

$\begin{cases}
2xz\cos(\theta)+3x^{2}\sin(\theta)=\lambda x \\
2yz\cos(\theta)-\sin(\theta)=\lambda y \\
(x^{2}+y^{2}+3z^{2})\cos(\theta)=\lambda z\\
cos(\theta)z(x^{2}+y^{2}+z^{2})=sin(\theta)(y-x^3)
\end{cases}$

\vspace{0.3cm}

\noindent Making some calculations you will get only the trivial solution.

\vspace{1cm}

{\textit{Acknowledgements}: This article was completed during my
visit to Northeastern University, Boston, USA, from February 16, 2006 - June 25, 2006.
I would like to thank Professor Terence T. Gaffney for
making my visit possible and for many helpful conversations during that
time.}

\vspace{0.1cm}


\begin{thebibliography}{99}

\bibitem[ACS]{acs} A. Fernandes, C. Humberto Soares, R.dos Santos, {\em Topological Triviality of
family of function and sets,} Rocky Mountain Journal of Mathematics, vol.36 , 4, pp.1235--1247, 2006

\bibitem[BK]{kbsk} K. Bekka and S. Koike, {\em The Kuo condition, an Inequality of Thom's Type and
$(c)$-Regularity,} Topology, vol.37 , 1, pp.45--62, 1998.

\bibitem[B1]{kb1} K.Bekka, {\em Regular Quasi-Homogeneous Stratifications}, Stratification,
singularities and diferential equations II: Stratifications and
Topology of Singular Space,  Travaux en cours 55, Hermann, 1997,
1--14.

\bibitem[B2]{kb2} K.Bekka, {\em Regular Stratification of Subanalytic Sets,}
Bull London Math. Soc., 25, 1993, 7-16.

\bibitem[Ja]{ja} A.Jaquemard, {\em Fibrations de Milnor pour des applications r\'eelles,}
Boll. Un. Mat. Ital., vol.37, 1, pp.45--62, 3-B,1989.

\bibitem[Ja1]{ja1} A. Jacquemard, {\em Th\`ese 3\`eme cycle Universit\'e de Dijon,} 1982.

\bibitem[Mi]{mi} J.Milnor, {\em Singular points of complex hypersurfaces,}
Ann. of Math. Studies 61, Princeton University Press, 1968.

\bibitem[RS]{rs} M.A.S.Ruas and R. Santos, {\em Real Milnor
Fibrations and $(C)-$regularity,} Manuscripta Math., 117, (2005),
no. 2, 207--218.

\bibitem[RSV]{rsv} M.A.S.Ruas, J.Seade and A.Verjovsky {\em On Real Singularities with a
Milnor Fibration,} Trends in Singularities, Birkhauser, 2002.

\bibitem[S]{s2} J.Seade, {\em Open Book Decompositions Associated to Holomorphic
Vector Fields,} Bol.Soc.Mat.Mexicana(3), vol.3 , pp.323--336,
1997.

\end{thebibliography}
\end {document}